\documentclass[11pt, twoside, letterpaper]{amsart}
\usepackage{amsmath,amsthm,amsfonts,amssymb,epsfig}
\usepackage{setspace}
\usepackage{stmaryrd}
\usepackage{amsmath,amssymb,amsthm,euro} 
\usepackage{bbm}
\usepackage{stmaryrd}
 \usepackage{tensor}

\usepackage{fancyhdr}
\newcommand\F{\mathcal F}
\newcommand\I{\mathbbm{1}}
\renewcommand{\P}{\mathbb{P}}
\newcommand{\Q}{\mathbb{Q}}

\newcommand{\M}{\mathcal M}

\newcommand{\D}{\mathcal D}

\renewcommand{\rho}{\varrho}

\renewcommand\i{\infty}

\newcommand\N{\mathbb{N}}
\newcommand\E{\mathbb{E}}
\newcommand\R{\mathbb{R}}
\newcommand\Z{\mathbb{Z}}

\newcommand\Lone{{L^1 (\P)}}

\newtheorem{theorem}{Theorem}

\newtheorem{corollary}[theorem]{Corollary}

\newtheorem{lemma}[theorem]{Lemma}

\theoremstyle{definition}

\theoremstyle{example}

\newcommand\A{\mathcal{A}}
\renewcommand\t{\tau}
\newcommand\s{\sigma}
\newcommand\w{\omega}

\newcommand\e{\varepsilon}
\newcommand\B{\mathcal{B}}
\renewcommand\L{\lambda}

\begin{document}

\title{On a dyadic approximation of predictable  processes  of finite  variation}
\author{Pietro Siorpaes}
%\date{26/06/1013}
\maketitle
\begin{abstract}
We show that any c\`adl\`ag predictable process  of finite variation is an a.s.  limit of \emph{elementary} predictable processes;  it  follows that predictable stopping times  can be approximated `from below' by predictable stopping times which take finitely many values.
We then obtain as corollaries two classical theorems: predictable stopping times are announceable, and	 an increasing process is predictable iff it is natural.
\end{abstract}

%AMS 2010 classification
%60G07 General theory of processes
%60G40  	Stopping times
%60G44  	Martingales with continuous parameter
%\keywords{ predictable; natural, compensator, submartingale, stopping time, announceable.}

We recall that a process $S=(S_t)_t$ is called \emph{of class D} if the family of random variables $(S_\t)_\t$, where $\t$ ranges through all stopping times, is uniformly integrable.
If  $S=(S_t)_{t\in [0,1]}$ is a  submartingale of class D, then it has a unique Doob-Meyer decomposition $S = M + A$, where $M$  is a uniformly integrable martingale and $A$ is a predictable increasing integrable process starting from zero, called the compensator  of $S$.
One can give constructive proofs of  the existence of the Doob-Meyer decomposition by taking limits of the discrete time Doob decompositions $(M^n_t+A^n_t)_{t\in \D_n}$ of the sampled process $(S_t)_{t\in \D_n}$ relative to refining partitions $(\D_n)_n$. Indeed in \cite{Rao69} the compensator  is obtained as the  limit of the $A^n$'s in the $\sigma(L^1,L^\i)$-topology; more simply, even if in general these discrete time approximations do not converge in probability to $A$ for all $t$,  one can always build some forward convex combinations $\A^n$ of the $A^n$ such that   $\limsup_n \A^n_t=A_t$ a.s. for all $t$, as was shown in \cite{BeScVe12}.
It follows that  $ \A^n_t \to A_t$ a.s. as $n\to \infty $ along a subsequence which a priori depends on both $t$ and $\omega$; it is then natural to ask whether there is such a subsequence $(n_k)_k$ which works simultaneously for all $(t,\omega)$, so that    $\A_t^{n_k}$ converges to $A_t$ a.s. for all $t$ as $k\to \infty$; in this paper we show that this is indeed the case,   in particular proving that  any  predictable increasing process $A$ is a pointwise limit of  predictable increasing processes  $\A^n$  of the form
\begin{align*} 
\textstyle
\A^n=\I_{\{0\}}A_0+ \sum_{k=1}^{2^n} \I_{( \frac{k-1}{2^{n}} , \frac{k}{2^{n}} ]} \A^n_{ \frac{k}{2^{n}} } \, .
\end{align*}

From this `dyadic' approximation, by time change it easily follows that predictable stopping times can be approximated `from below' by predictable stopping times which take finitely many values; this being a predictable analogue of the simple fact that any stopping time is the limit of a decreasing sequence of stopping times, each taking values in a finite set.

We notice how these results can be used to provide an alternate derivation of the following well known theorems: predictable stopping times are announceable, and an increasing process is predictable iff it is natural (we show this passing to the limit the analogous discrete time statement). 

To prove the convergence results of \cite{BeScVe12} and our main theorem, use is made of the existence of a sequence of stopping times which  exhaust the jumps of a c\`adl\`ag  adapted  process. We prove this classic result without using the deep debut and section theorems, by showing explicitly that the jumps times, and the `first-approach time', of a c\`adl\`ag (predictable) adapted   process are (predictable)  stopping times; our proofs are elementary, and hold even if the filtration does not satisfy the usual conditions.
Moreover, we show how a simple variant of this result can be used
to characterize  the continuity 
of local martingales and of the compensator  of special semimartingales.

The rest of the paper is organized as follows: 
in Section \ref{The main results}  we  introduce some definitions and conventions, and we
state our results on the approximation of predictable processes of finite  variation and of predictable stopping times, and we prove the second one.
In Section \ref{PFA} we discuss the equivalent characterizations of  predictable stopping times. 
In Section \ref{Predictable stopping times} we state and prove some classical results on predictable stopping times.
  In Section \ref{How to approximate the compensator} we prove our previously-stated main result on the convergence of the  dyadic approximations.
In Section \ref{Predictable processes are natural, and vice versa}
we show that an increasing process is predictable iff it is natural. Finally, in Section \ref{Consequences for special semimartingales} we derive some  corollaries about special semimartingales.

\section{The main results}
\label{The main results}
In this Section, after introducing some definitions and conventions, we state our results on the approximation of predictable  processes  of integrable variation  and of predictable stopping times, and we prove the second one.

In this article we will consider a fixed filtered probability space  $(\Omega,\mathcal{F},\mathbb{F},P)$ and  we assume that the filtration $\mathbb{F}=(\mathcal{F}_t)_{t\in [0,\infty]}$ satisfies  the usual conditions of right continuity and saturatedness.  By convention, the $\inf$ of an empty
 set will be  $\infty$.
Inequalities are meant in the weak sense, so $t_n\uparrow t$ means $t_n\leq t_{n+1}\leq t$ and $t_n \to t$, `increasing' means `non-decreasing' etc. 
  We will say that a process $A$ is  increasing (resp. of finite variation) if it is adapted and $(A_t(\omega))_t$ is  increasing (resp. of finite variation) for $\P$ a.e. $\omega$. A process $X$ is called integrable if $\sup_t |X_t |\in \Lone$.
 A property of a process $S$ (integrability, martingality, boundedness etc.) is said to hold locally if there is a sequence of stopping times $\t_n\uparrow \i$ s.t., for each $n\in \N$, $S^{\t_n}\I_{\{\t_n>0\}}$ satisfies the property.
All  local martingales we will deal with are assumed to be c\`adl\`ag. Given a c\`adl\`ag process $X$, we set $X_{0-}:=X_0$  and $X_{-}:=(X_{t-})_t$, and define   $\Delta X_t:=X_t-X_{t-}$ for $t\in [0,\infty)$ and $\Delta X_t=0$ for $t=\infty$.
We will call a process predictable if it is measurable with respect to the sigma algebra generated on $[0,\infty)\times \Omega$  by the c\`ag adapted processes; a  stopping time $\t$ will be called  predictable if $\I_{[\t,\infty)}$ is predictable.
We denote with $\D_n$ the set $\{k/2^n:k=0,\ldots, 2^n \}$ of dyadics of order $n$ in $[0,1]$, and with $\D=\cup_{n\in \N} \D^n$  the set  of all dyadics in $[0,1]$.

To state and prove our main theorem we need the following non-standard definitions: we will say that a process $B$ is   \emph{$\D_n$-predictable} if it is of the form
\begin{align}
\label{step} 
\textstyle
B=\I_{\{0\}}B_0+ \sum_{s\in\D_n\setminus \{0\}} \I_{( s-2^{-n} ,s]} B_{s} \, ,
\end{align}
where  $B_{0} $ is  $ \F_{0}$-measurable and $B_{s} $ is  $ \F_{ s-2^{-n}}$-measurable  for every $s\in\D_n\setminus \{0\}$; given $a\in (0,1]$, we define $\D_k(a):=\max\{ s\in \D_k: s<a\}$.

\begin{theorem}
\label{SumDoob}
If $A=(A_t)_{t\in[0,1]}$ is a c\`adl\`ag predictable process with finite variation, there exist a subsequence $(N_n)_n$ of $(N)_{N\in \N}$ and, for each $n\in\N$, a   $\D_{N_n}$-predictable process $\A^n$
such that  $\exists \lim_n \A^{n}_t=A_t$ a.s. for all $t\in[0,1]$ and $\A^n_0=A_0$. If $A$ is increasing then each $\A^n$ can be chosen to be increasing, and if $A$  has integrable variation then $(\A^n)_n$ can be chosen so that $|var(\A^n)_{1}| \leq h$ for all $n\in \N$ for some $h\in \Lone$.
\end{theorem}

Because of the relationship between increasing adapted  processes and time-changes, it is now easy to prove  the following.
 \begin{theorem}
\label{DyadicStop}
If $\tau$ is a predictable stopping time, one can construct for each $n\in \N$ a predictable stopping time $\s_n$  with values in a finite set and such that $\s_n\to \t$,  $\s_n=0$ on $\{\t=0\}$ and,  if $\omega \in \{\t>0\}$, there exists $n_0(\omega)$ s.t. $\s_n(\omega)<\t(\omega)$  for all  $n\geq n_0(\omega)$. 
\end{theorem} 
\begin{proof}
As $[0,1]$ is homeomorphic to $[0,\infty]$, we can assume w.l.o.g. that 
 $\t$ has values in $[0,1]$. 
Apply Theorem \ref{SumDoob} to  $A:=\I_{[\t,1]}$ to obtain $N_n$ and increasing $\A^n$. Define  $\s_n:=\inf\{ t\in [0,1]: \A^n_t\geq 1/2\}\wedge 1$, and notice that since $\A^n$ is  $\D_{N_n}$-predictable, trivially 
 $\s_n$ is a predictable stopping time with values in $\D_{N_n}$, and $\s_n=0$ on $\{\t=0\}$ since $\A^n_0=A_0$.
Since  $\lim_n \A^n_\t=A_\t =1$ for $\P$ a.e. $\omega$, there exists $n_0=n_0(\w)$ s.t. $\A^n_\t >1/2$ for all $n\geq n_0$; so on $\{\t>0\}$,  since $\A^n$ is constant on the interval $(\D_{n}(\t), \D_{n}(\t)+2^{-n}]$ which contains $\t$,
necessarily $\s_n\leq \D_{N_n}(\t)<\t$ holds for all $n\geq n_0$. 

 Moreover if $\e>0$, $\lim_n \A^n_{\t-\e}=A_{\t-\e}=0$ a.s. on $\{\t-\e>0\}$, and so there exists $n_1=n_1(\w)$ s.t. $\A^n_{\t-\e}\leq 1/4$ for all $n\geq n_1$; it follows that a.s. $\liminf_n \s_n\geq \t -\e$, and so a.s. $ \t  \leq  \liminf_n  \s_n\leq\limsup_n \s_n\leq \t$ on $\{\t>0\}$. Now just
re-define $\s_n$ as $\max\{0,\t -1/n\}$ on the null set where either $\lim_k \s_k =\t $ fails or  $\s_n\geq\t>0$.
\end{proof}

Notice that in the previous proof the re-defined $\s_n$'s are still predictable, because any measurable process indistinguishable from zero is predictable when the filtration satisfies the usual conditions (see \cite[Lemma 13.8]{RoWi00}).

\section{Predictable, fair and announceable stopping times}
\label{PFA}
In this section we notice that  Theorem \ref{DyadicStop} immediately implies that predictable stopping times are announceable, and then discuss other proofs of this important result which are found in the literature.

If $\t_n$ is an increasing  sequence of stopping times converging to $\t$ 
and such that  $\t_n <\t$  on $\{\t>0\}$ for all $n$,  we will say that $\t_n$ \emph{announces} $\t$; a stopping time $\t$ for which such an announcing sequence exists is called \emph{announceable}.
Trivially  announceable stopping times are predictable: if $\t_n$ announces $\t$, the process  $\I_{[\t,\i)}$ 
is the   pointwise  limit of the c\`ag adapted 
processes $\I_{\{0\}} \I_{\{\t_n=0\}} +  \I_{(\t_n,\i)}$; we now show that the opposite holds too, and can even be strengthened.

\begin{corollary}
\label{PrImplAnn}
 Any predictable stopping time can be announced by a sequence of  predictable stopping times.
\end{corollary}
\begin{proof}
If $(\s_n)_n$  are as in Theorem \ref{DyadicStop} then $\inf_{k\geq n} \s_k$ is attained, and so the increasing sequence of stopping times $\t_n:=\inf_{k\geq n} \s_k$ satisfies  $\I_{[\t_n,\infty)}=\inf_{k\geq n}  \I_{[\s_k,\infty)}$; thus, each  $\t_n$ is a predictable stopping time. Since trivially    $ \lim_n  \t_n=\t$, $\t_n<\t$ on $\{\t>0\}$ and $\t_n=0$ otherwise, the thesis follows. 
\end{proof} 

Notice that one could alternatively first prove Corollary \ref{PrImplAnn}, and then easily derive from it Theorem  \ref{DyadicStop}. Indeed, let 
 $(\t_n)_n$ be predictable stopping times announcing $\t$
and $(\t_n^k)_k$ be a decreasing sequence of predictable stopping times,  each taking values in a finite set, s.t. $\lim_k \t_n^k=\t_n$ and  
$\t_n^k=0$ on $\{\t_n=0\}$.
Then,  if    $(k_n)_n$ is a subsequence s.t. $\P( \tau_n^{k_n} \geq \t>0)<1/2^n$, Theorem  \ref{DyadicStop} follows taking first $\s_n:= \tau_n^{k_n} $, and then re-defining $\s_n$ as $\max\{0,\t -1/n\}$ on the null set where $\tau_m^{k_m} \geq \t>0$ happens for infinitely many $m$'s.

Corollary \ref{PrImplAnn} is typically proved  using yet another useful condition equivalent to being announceable; following \cite{RoWi00}, we will say that a  stopping time $\t$ is  \emph{fair} if $\E[M_{\t}]=\E[M_{\t-}]$ holds for every bounded martingale $M$ (where we set  $M_{\infty}:=M_{\infty-}:=\lim_{t\to \infty} M_{t}$).
Since by the optional sampling theorem $M_{\t_n}=\E[M_{\t}|F_{\t_n}]$ holds for every uniformly integrable martingale $M$ , taking expectations and passing to the limit  shows that announceable stopping times are fair (more, it shows that if $\t$ is announceable and $M$ is a uniformly integrable martingale  then $M_{\t-} \in \Lone$ and  $\E[M_{\t-}]=\E[M_{\t}]$).    The opposite implication is also true, so being predictable, being fair and being announceable are equivalent  conditions for a stopping time.
 For a proof of the fact that fair stopping times are announceable we refer to \cite[Chapter 6, Theorem 12.6]{RoWi00}, as we have nothing to add to this implication.
 Knowing that fair stopping times are announceable, to conclude the proof of all the equivalences typically one shows directly that predictable stopping times are fair (in the present paper, this follows instead from Corollary \ref{PrImplAnn}  since, as explained above, announceable stopping times are fair). This fact is often derived as a consequence of the (difficult) section theorem for predictable sets; one can however find  in \cite[Chapter 6, Theorem 12.6]{RoWi00} a direct proof which does not use the section theorems themselves, but does involve ideas from their proofs, which are  essentially based on Choquet's capacity theorem.
Another possibility is to proceed as \cite{MePe80} and give a proof of the Doob-Meyer decomposition which shows inter-alia that increasing predictable processes are natural; applying this to the predictable process $A:=\I_{[\t,1]}$ shows that predictable stopping times are fair (in particular in this paper, instead of  proving Theorem \ref{DyadicStop} directly, we could see it as an immediate corollary of  Theorem \ref{PredSoNat}).
Yet another way is to proceed as in \cite{LoBlog}; this proof, although intuitive, uses the theory of integration with respect to general predictable bounded integrands, as well as the Bichteler-Dellacherie theorem and Jacod's countable expansion theorem.

\section{Predictable stopping times}
\label{Predictable stopping times}
In this section we  prove that a number of commonly used hitting times are (predictable) stopping times when the underlying process is c\`adl\`ag and adapted (predictable), and we
use this to show that one can exactly exhaust the jumps of a c\`adl\`ag adapted (predictable)  process with sequence of (predictable) stopping times.
To precisely state this fact, we recall that 
 $\llbracket \s \rrbracket:=  \{  (t,\omega)\in [0,\infty)\times \Omega: t=\s(\omega) \}$ denotes the \emph{graph of a}  \emph{stopping time} $\s$.
If $X$ is c\`adl\`ag we say that a sequence of stopping times 
 $(\s_n)_{n\in \N}$  \emph{exactly} exhausts the jumps of $X$ in $B\subseteq \R$ if $\{ \Delta X \in B \} = \cup_{n} \llbracket \s_n \rrbracket$  and $ \llbracket \s_n \rrbracket\cap  \llbracket \s_m \rrbracket=\emptyset $ whenever $n\neq m$ (i.e. $\s_n \neq \s_m$ on $\{\s_m<\infty\}$ for $n\neq m$).
We will say that    $(\s_n)_{n\in \N}$ is \emph{ strictly increasing to $\infty$} if $ \s_n <  \s_{n+1} $  on $\{\s_n <\infty \}$  and $\s_i \to \infty$ as $i\to \infty$.
Given $B\subseteq \R$ we set $d(x,B):=\inf\{|x-y|:y\in B \}$.

\begin{theorem}
\label{ExPredJumps}
If $X=(X_t)_{t\geq 0}$ is a c\`adl\`ag adapted process,  $C_k$ are closed in $\R$ and $0\notin F=\cup_k C_k$
 then  there exist  a sequence of stopping times $(\s_n)_n$ which exactly exhausts the jumps of $X$ in $F$, and if $X$ is predictable then each $\s_n$ can be chosen to be predictable.
Moreover, if  $d(0,F)>0$ then $(\s_n)_n$ can be chosen to be  strictly increasing to $\infty$.
\end{theorem}

The rest of this section is devoted to giving an \emph{elementary} proof of the classical Theorem \ref{ExPredJumps} and of the following lemma (which we only  use in Section \ref{Consequences for special semimartingales}); thus, the reader interested in new results may safely decide to jump directly to Section \ref{How to approximate the compensator}.

\begin{lemma}
\label{ApproachTimeLemma}
If $X$ is c\`adl\`ag adapted and $C$ is closed, the first-approach time 
\begin{align}\label{1appr}
\s:=\inf\{ t \geq 0:  X_t  \in C \text{ or } X_{t-}  \in C \} 
\end{align}
is a  stopping time, and it is predictable if $X$ is predictable.
\end{lemma}
 
We remark that all the results in this section hold, with exactly the same proof, if the process  $X$ has values not in $\R$ but in a generic topological vector space $Y$ which supports a translation invariant distance $d$ that generates the topology, in which case we assume that  $Y$ is endowed with the Borel sigma algebra $\B(Y)$.
Notice that the function $d(\cdot, B):=\inf\{d(x,y):y\in B\}$ is Lipschitz (with constant $1$), so it is Borel measurable. We will use without further mention the fact that if $B$ is closed and $y_k\to y$ then $d(y_k,B)\to 0$ implies $y\in B$.

\begin{proof}[Proof of Lemma \ref{ApproachTimeLemma}]
Let $t_n \downarrow t$ be s.t. either $ X_{t_n}   \in C $ or $X_{t_n-}  \in C $. If $t_n=t$ then either $ X_{t}   \in C $ or $X_{t-}  \in C $, and if $t_n \downarrow t$ and $t_n>t$ for all $n$ then, whether $ X_{t_n}   \in C $ or $X_{t_n-}  \in C $, necessarily $ X_{t}   \in C $;
thus, the infimum in \eqref{1appr} is attained.
From this and the compactness of $[0,t]$  it follows that $\{\s \leq t\} =L$, where
\begin{align}
\label{ApproachTime}
  \{ X_t \in C \} \cup \bigcap_{n\in \N} \bigcup_{q\in \Q \cap [0,t)}   \{ d(X_q,C)<1/n  \}=:L \quad \text{ belongs to }  \F_t ;
\end{align}
thus, $\s$ is a stopping time.

Now suppose that $X$ is predictable. Since the infimum in \eqref{1appr} is attained, if $t\leq \s$ then $\s= t$
 iff   either $ X_{t}  \in C $ or  $X_{t-} \in C $; in other words 
\begin{align}
\label{Crucial1}
\I_{\{\s\}}= \I_{[0,\s]} \Big(\I_C( X ) \vee \I_C(X_{-})\Big) .
\end{align}
The c\`ag processes $\I_{[0,\s]} $ and $X_{-}$  are adapted and thus predictable, so  \eqref{Crucial1} implies that  $\I_{\{\s\}}$ is predictable and so also   $\I_{[\s,\infty)}=1-\I_{[0,\s]} +\I_{\{\s\}}$ is  predictable. 
\end{proof}

 Although we could, similarly to \cite[Chapter 1, Proposition 1.32]{JacodShir:02},  prove part of Theorem \ref{ExPredJumps}
using Lemma \ref{ApproachTimeLemma}
(making use of the concept of the sigma algebra $\mathcal{F}_{\t -}$), we find it more natural to study directly the jumps times of $X$ as follows.

\begin{lemma}
\label{PredJump}
 If $\t$ is a stopping time, $X$ is c\`adl\`ag adapted, $C_n$ are closed sets and $F=\cup_n C_n$  satisfies  $d(0,F)>0$  then
\begin{align}\label{Jump}
\s:=\inf\{ t > \t: \Delta X_t  \in F \} 
\end{align}
is a  stopping time s.t. $\s>\t$ on $\{ \t<\infty\}$, and $\s$ is predictable if $X$ is predictable.
\end{lemma}
\begin{proof}
We will use the fact that, since $X$ is c\`adl\`ag, for any compact interval $J$ the set $ \{ t\in J:  d(\Delta X_t,0) \geq d(0,F) \} $ is finite, so the set $D:=\{ t > \t: \Delta X_t  \in F \} $  is discrete; in particular, the $\inf$ defining $\s$ is attained, so $\s>\t$ on $\{ \t<\infty\}$.
That $\s$ is a stopping time follows from the identity $\{\s \leq t \}=L\in \F_t$, where
\begin{align*}
L:=  \bigcup_{q \in \Q \cap (0,1)}  \Big(  \{  \t \leq qt\} \cap \Big( \bigcup_{n\geq 1} \bigcap_{k\geq 1} \bigcup_{(u,s)\in A_k^q(t)} \{ d(X_s-X_u,C_n)<\frac{1}{k}   \} \Big) \Big) ,
\end{align*}
and $A_k^q(t) $ is the countable set 
$$A_k^q(t):=\{ (at,bt): a,b \in \Q \cap (q,1], a<b<a+1/k   \} .$$
To prove $\{\s \leq t \}=L$, consider that, 
 since $D$ is discrete,  $\s\leq t$ iff $t >\t$ and   there exist  $s\in (\t,t] $ and  $n\in \N$
s.t. $\Delta X_{s}  \in C_n $.
Thus, $\s\leq t$ iff  there exist $q \in \Q \cap (0,1)$,  $n\in \N$, and sequences $u_k \uparrow s$ and $s_k \downarrow s$  s.t.  $  tq \geq \t$, $u_k<s$, $(u_k,s_k)\in A_k^q(t)$ and
\begin{align}\label{differences}
 d(X_{s_k}-X_{u_k},C_n)<1/k  \quad  \text{ for all } k.
\end{align}
  This shows that $\{\s \leq t \}\subseteq L$, and that to prove the opposite inequality
given $q \in \Q \cap (0,1)$ and $(u_k,s_k)\in A_k^q(t)$  such that  $  tq \geq \t$ and \eqref{differences} hold, we only need to show that we can replace $(u_k,s_k)$ with some $(\hat{u}_k,\hat{s}_k)$ satisfying the same properties and additionally s.t. $\hat{u}_k<s$, $\hat{u}_k \uparrow s$ and $\hat{s}_k \downarrow s$ for some $s$.
This is easily done: by compactness  there exists a subsequence $(n_k)_k$ s.t. $u_{n_k}$ (resp. $s_{n_k}$) is converging to some $u$ (resp $s$) and w.l.o.g. the convergence is monotone; 
 since $u_k<s_k<u_k+1/k $ necessarily $u=s$, and since \eqref{differences} implies  $ \liminf_k  d(X_{s_k}-X_{u_k},0)\geq d(0,F)>0$ necessarily $s_{n_k}$ must be decreasing and  $u_{n_k}$   increasing  and s.t. $u_{n_k}<s$ (otherwise $ \exists \lim_k  X_{s_k}-X_{u_k}=0$); thus, we can choose 
$(\hat{u}_k,\hat{s}_k):=(u_{n_k},s_{n_k})\in A^q_{n_k}(t)\subseteq A_k^q(t)$ as it also satisfies \eqref{differences}.

Now suppose that $X$ is predictable and notice that, since $D$ is discrete, if $\t< t\leq \s$ then $\s= t$ iff $\Delta X_{t}  \in F $; in other words 
$\I_{\{\s\}}= \I_{(\t,\s]} \I_{F}(\Delta X) ,$
which implies that $\s$ is predictable  (just as \eqref{Crucial1} does in Lemma \ref{ApproachTimeLemma}).  \end{proof}
\begin{proof}[Proof of Theorem \ref{ExPredJumps}]
If $d(0,F)>0$, let $ \s_{-1}:=0$  and define recursively $(\s_{k})_{k\in \N}$ by setting   $ \s_{k+1}:=\inf\{ t > \s_{k}:  \Delta X_t  \in F\}$.
By Lemma \ref{PredJump} each $\s_{k}$ is a stopping time, and a predictable one if $X$ is predictable. 
Since  $ \s_{n} < \s_{n+1} $  on $\{\s_n <\infty \}$, and since for any compact interval $J$ the set $ \{ t\in J:  d(\Delta X_t ,0) \geq d(0,F) \} $ is finite,  
$(\s_{k})_{k\in \N}$ exhausts the jumps of $X$ in $F$ and it is strictly increasing to $\infty$.

For general $F$ s.t. $0\notin F$, we  reduce to the previous case by using the annullus
 $D_n:=\{y: d(y,0)\in (2^n,2^{n+1}]\}$. Notice that  $D_n$ can be written as the union of countably many closed sets, and so also can $F\cap D_n$.
 Now, given $n\in \Z$, set $ \s_{-1}^n:=0$ and define recursively $(\s^{n}_k)_{k\in \N}$ by setting   $ \s_{k+1}^n:=\inf\{ t > \s_{k}^n:  \Delta X_t  \in F\cap D_n \}$, so that  $$\{ \Delta X \in F \}=\cup_n \{  \Delta X  \in F\cap D_n \}=\cup_{k,n} \llbracket \s^n_k \rrbracket .$$
Moreover,  since
 $ \s^n_{k+1} > \s^n_{k} $  on $\{\s^n_k <\infty \}$, and since $D_n$ and $D_m$ are disjoint for $n \neq m$ and $\Delta X_{\s^n_{k}} \in  F\cap D_n $  on $\{ \s^n_{k} <\infty\}$  for every $k$, it follows that $ \s^i_{j} \neq \s^n_{k} $  on $\{ \s^n_{k} <\infty\}$ if $(i,j)\neq (n,k)$, so enumerating the countable family $(\s^{n}_k)_{k, n\in \Z}$  we get a sequence which exactly exhausts the  jumps of $X$ in $F$.
\end{proof}

\section{How to approximate the compensator}
\label{How to approximate the compensator}
In this section we prove Theorem \ref{SumDoob}; to do this, we revisit the proof of the existence of the Doob-Meyer decomposition given \cite{BeScVe12}, and we strengthen it as to obtain that $\A^n_t\to A_t$ a.s. for all $t$ along a subsequence.
For didactical reasons we prefer to present below the whole proof, rather than explaining how to modify the one given in \cite{BeScVe12}.

To obtain convergence at a given stopping time, we will use the following lemma, which is reminiscent of \cite[Lemma A.2]{LaZi07}, 
and whose point is that the subsequences (in the assumption and in the conclusion) are not allowed to depend on $\omega$.
\begin{lemma}
\label{subseq}
Let $f,g,(f^n)_{n}, (g^n)_{n}$ be random variables  in  $\Lone$ that satisfy 
$$0\leq f^n \leq g^n ,\quad g^n \to g  \text{ in } \Lone \,  , \quad 
  \lim_{n\to\infty}\E[f^n] = \E[f]. $$ 
Assume moreover that for every  subsequence $(n_i)_i$
\begin{align}
\label{AlongSubseq}
 \limsup_{i\to\infty}
f^{n_i}(\omega) = f (\omega) \quad \text{ for }\, \,  \P  \, \text{a.e. } \omega . 
\end{align}
Then, there exists $h\in \Lone$ and a  subsequence $(n_i)_i$ such that $f^{n_i}\leq h$ for all $i$ and  $(f^{n_i})_{i}$ converges almost surely to $f$ as $i\to \infty$.
\end{lemma}
\begin{proof}
Passing to a subsequence (without relabeling) we get that $ || g^n -g ||_{\Lone}\leq 2^{-n}$, thus  the random variable $h:=g+\sum_n | g^n -g |$ is integrable and 
 dominates the sequence $(f^n)_n$.
By  the dominated convergence theorem and \eqref{AlongSubseq} it follows that
\begin{equation}
\label{Esupf} 
\lim_{n\to\infty} \mathbb{E}\left[\sup_{m\ge n}f^m\right] =
\mathbb{E}[f] .
\end{equation}
The assumption $ \mathbb{E}[f] = \lim_{n}\mathbb{E}[f^n]$ and \eqref{Esupf} imply that  $ h^n:=f^n-\sup_{m\ge n} f^m$ converges to $0$ in $\mathbb{L}^1$ (since $h^n\leq 0$).
 We can then extract a further subsequence (not relabeled) 
 such that $ h^{n}$ converges to $0$ $\P$ a.s. Thus, thanks to the monotonicity of $ (\sup_{m\geq n}f^m)_n$, also $(f^{n})_n$ also converges a.s., and then \eqref{AlongSubseq} implies that its limit is $f$.
\end{proof}

\begin{proof}[Proof of Theorem \ref{SumDoob}]
The identity
  $var(A)_t=var(A)_{t-}+|A_t -A_{t-}|$ shows that    $var(A)$ is  predictable (since $var(A)_{-}$ and $  A_{-}$ are adapted and c\`ag). Thus $A^{\pm}:=(var(A)\pm A)/2$ are predictable increasing and satisfy $A=A^+ -A^-$, so we can assume  w.l.o.g. that $A$ is increasing.  Moreover, by passing to an equivalent measure we can assume w.l.o.g. that $A$ is integrable.

 If  $A_0=0$ set $S:=A$, which trivially is a submartingale of class D.
Let $(M^n_t+A^n_t)_{t\in \D_n}$  be the  discrete time Doob decomposition of the sampled process $(S_t)_{t\in \D_n}$,
and  extend $M^n$ and $A^n$ to $[0,1]$  setting 
$$M^n_t:=\E[M^n_1|\F_t] \quad  \text{ and } \quad A^n_t:= A^n_{k/2^n}  \,\,  \text{ for } t\in ((k-1)/2^n,k/2^n] ;$$
 then it follows from \cite[Lemma 2.1 and 2.2]{BeScVe12} that  there exist $\hat{M}\in \Lone$ and convex weights $\L_n^n, \ldots, \L_{N_n}^n$ such that 
$\M^n:=\L_n^n M^n + \ldots + \L_{N_n}^n M^{N_n}$ satisfies $\M^n_1\to \hat{M}$ in $L^1$.
Now define 
\begin{align}
\label{DefOfA}
M_t:=\E[\hat{M}|\F_t] , \quad B:=S-M ,  \quad  \A^n:=\L_n^n A^n + \ldots + \L_{N_n}^n A^{N_n} .
\end{align}
We take of course the c\`adl\`ag versions of the martingales $M^n$ and $M$; in particular, $B$ is c\`adl\`ag. We now want to show that,  a.s. for all $t\in[0,1]$,   $\exists \lim_i \A^{n_i}_t=B_t$ for some subsequence $(n_i)_i$ (which does not depend on $t$ nor $\omega$); this would show\footnote{Because $\A^n$ is adapted and c\`ag, and any measurable process indistinguishable from zero is predictable when the filtration satisfies the usual conditions (see \cite[Lemma 13.8]{RoWi00}).} that $B$ is predictable, so $S=M+B$ would be a Doob-Meyer decomposition of $S=0+A$, and thus $B=A$ by the uniqueness of the decomposition (which follows from  \cite[Lemma 22.11]{Kall97}). 

Since  $\M^n_{1} \to \hat{M}=M_{1}$ in $\Lone$,  by Jensen inequality and  the optional sampling theorem we get that, for every  $[0,1]$-valued stopping time $\t$,  $S_{\t}-\M^n_{\t}$ converges to $ S_{\t}-M_{\t} = B_{\t}$ in $L^1$; in particular,  
since $\A^n_t=S_{t}-\M^n_{t}$ holds for $t\in \D_n$, we get that 
$\A^n_t \to B_{t}$ in $L^1$ for all $t\in \D$. 
 Passing to a subsequence (without relabeling), we can also obtain that $ ||\A^n_1 - B_{1}||_{\Lone}\leq 2^{-n} $ and $\A^n_{t} \to B_{t}$ a.s. for all $t\in \D$. It follows that $B$ is a.s. increasing on $\D$, and so by right-continuity also on $[0,1]$, and the random variable $h:=B_{1}+\sum_n | \A^n_1 -B_{1} |$ is integrable and  dominates the sequence $(\A_{1}^n)_n$.
 
We remark that, since the equality $\A^n_t=S_{t}-\M^n_{t}$ generally fails if $t\notin \D_n$, it is unclear for now  if,  given a $[0,1]$-valued  stopping time $\t$, we can also get $\A^{n}_\t \to B_{\t}$ a.s.; we will now explain how to obtain this by passing to a subsequence.
We only need to show that   $\exists \lim_{i} \E[\A^{n_i}_\t]= \E[B_\t]$ and  $\limsup_i \A^{n_i}_\t= B_\t$ a.s. for every  subsequence $(n_i)_i$; indeed, applying Lemma \ref{subseq} to $f^n=\A_{\t}^n$, $f=B_{\t}, g^n=\A_1^n$ and $g=B_1$ would then yield a subsequence $(\tilde{n}_i)_i$ such that $\lim_i\A_{\t}^{\tilde{n}_i} = B_{\t}$ a.s..
Take then an arbitrary subsequence $(n_i)_i$, and recall that $\A^n,B$ are increasing and  $\A^n_t \to B_{t}$ a.s. and in $L^1$ for all $t\in \D$. It follows that $\limsup_i \A^{n_i}_\t\leq B_\t$, and that applying Fatou's lemma to $(\A^{n_i}_1-\A^{n_i}_\t)_i$ gives 
$$\liminf_i \E[A^{n_i}_\t]\leq \liminf_i \E[\A^{n_i}_\t]\leq \limsup_i \E[ \A^{n_i}_\t]\leq  \E[\limsup_i \A^{n_i}_\t]\leq\E[B_\t] . $$ 
Thus, to conclude the existence of $(\tilde{n}_i)_i$ such that $\lim_i\A_{\t}^{\tilde{n}_i} = B_{\t}$ a.s. it is enough to show that $\exists \lim_n \E[A^{n}_\t]=\E[B_\t] $. This is easy: since $S$ is of class D, if $\theta_n:=\min \{t\in \D_n: t\geq \t\}$ then $\theta_n\downarrow \t$ and $A^n_\t=A^n_{\theta_n}$ so we get
$$\E[A^{n}_\t]=\E[A^{n}_{\theta_n}]=\E[S_{\theta_n}]-\E[M_0]\to \E[S_\t]-\E[M_0]=\E[B_\t] .$$

Now use Theorem \ref{ExPredJumps} to obtain $[0,1]\cup\{\infty\}$-valued\footnote{As we are working on the time interval $[0,1]$,  all stopping times have values in $[0,1]\cup\{ \infty\}$.}  stopping times $(\s_k)_k$
which exactly exhaust all the jumps of $B$ (i.e. the jumps of $B$ in $\R\setminus\{0\}$), and set $\t_k:=1\wedge \s_k$.
As shown above, there exists a subsequence $(n_i)_i$   such that $\lim_i\A_{\t_k}^{n_{i}} = B_{\t_k}$ a.s. for $k=1$. By the same token, passing to further subsequences (without relabeling) and using a diagonal procedure, we can find a subsequence $(n_i)_i$ such that $\lim_i\A_{\t_k}^{n_i} = B_{\t_k}$ a.s. simultaneously for all $k$.
Since $\A^n,B$ are increasing, $B$ is c\`adl\`ag and  $\A^n_t \to B_{t}$ for all $t\in \D$,  necessarily $\exists \lim_n \A^{n}_t=B_t$  if $B$  is continuous at $t$; since  $\exists \lim_i\A_{\t_k}^{n_i} = B_{\t_k}$  for all $k$ and 
$(\t_k)_k$ exhausts all the jumps of $B$, it follows that  $\exists \lim_i \A^{n_i}_t=B_t$ a.s. for all $t\in [0,1]$.
Since $\A^n_0=0$, this  concludes the proof in the case $A_0=0$; in the general case, apply the above to  $\tilde{A}=A-A_0\I_{[0,\i)}$ to obtain some $\tilde{\A}^n$ and $\tilde{h}$, and then set   $\A^n:=A_0\I_{[0,\i)}+ \tilde{\A}^n$ and $h:=|A_0|+ \tilde{h}$.
\end{proof}

\section{Predictable processes are natural, and vice versa}
\label{Predictable processes are natural, and vice versa}
In this section we prove in continuous time that an increasing process is predictable iff it is natural by passing to the limit the analogous discrete time statement, making use of 
Theorem \ref{SumDoob}. 

 \emph{In discrete time}, we will call \emph{increasing  process} 
an increasing sequence of integrable random variables $A=(A_n)_{n\in \N}$ such that  $A_0=0$. An increasing  process is called 
predictable if $A_{n+1}$ is $\F_{n}$-measurable for every $n\geq 0$, and is called 
natural if, for every bounded martingale $M=(M_n)_{n\in \N}$, we have $\E[M_n A_n]=\E[\sum_{k=1}^n M_{k-1} (A_k-A_{k-1})]$
for every $n\geq 0$.
In this setting it is trivial to prove that  an increasing processes is  predictable  iff it is  natural (see e.g. \cite[Chapter 1, Proposition 4.3]{KaSh88}).

When working on the time interval $[0,1]$, 
 a c\`adl\`ag increasing integrable process $A$ s.t. $A_0=0$ is called \emph{natural} if, for every bounded martingale $M$, 
$\E[M_1 A_1]=\E[\int_0^1 M_{s-} dA_s]$. 
Notice that in some books an equivalent definition is used: in discrete time it is trivial to prove that every increasing  process $A$ satisfies  $\E[M_n A_n]=\E[\sum_{k=1}^n M_{k} (A_k-A_{k-1})]$ for all bounded martingales $M$ and $n\in \N$, and that taking continuous time limits one immediately obtains that every c\`adl\`ag increasing  process $A$ satisfies  
$\E[M_1 A_1]=\E[\int_0^1 M_{s} dA_s]$ 
(see e.g.  \cite[Chapter 1, Lemma 4.7]{KaSh88}); thus $A$ is natural iff $\E[\int_0^1 \Delta M_{s} dA_s]=0$.

To pass to the continuous time limit, we will need the following approximation lemma and definitions.
Given functions $f,g$ and a partition $\pi=\{0=t_0\leq t_1\leq  \ldots \leq t_{n+1}=1\}$, set
\begin{align*}
 \textstyle 
f^{\pi}:=\sum_{i=0}^n f(t_{i}) \I_{[t_{i},t_{i+1})}+
f(1) \I_{\{1\}}  \, , \, \sum_{\pi} f \Delta g:=\sum_{i= 0}^n f(t_{i}) (g(t_{i+1})-g(t_{i})) .
\end{align*}
\begin{lemma}
\label{approxSumInt}
Given  $f,g,g_n:[0,1]\to \R$, with $g,g_n$ increasing and $f,g$ c\`adl\`ag, let $D$ be a dense subset of $[0,1]$ and  $(\pi^n)_{n\in \N}$ be partitions of $[0,1]$ which satisfy $ \pi_n \subseteq \pi_{n+1}$ and  $\cup_n \pi_n =D \supseteq 
\{\Delta f  \neq 0\}$.
Then $g_k(t)\to g(t)$ for all $t\in D$ implies that 
\begin{align}
\label{DiscreteConvCont}
\sum_{\pi_k} f\Delta g_k \to \int_{(0,1]} f(s-)  dg(s)=:\int f_{-} dg  \quad \text{ as } \, k\to \infty .
\end{align}
\end{lemma}

\begin{proof} %[Proof of Lemma \ref{approxSumInt}]
Since $f$ is right continuous at $0\in D$, for any $k\in \N$ the set $$A_k:=\{t\in(0,1]: \limsup_n \sup_{[0,t) } |f^{\pi_n} -f | <1/k\}$$
is non-empty. 
Its supremum $\bar{t}$ is attained,
 since $D$ is dense and $\exists f(\bar{t}-)$, and cannot be $<1$: otherwise, whether $\Delta f(\bar{t})=0$ or $\bar{t}\in \pi_n$ for big enough $n$,  the right continuity of  $f$ would imply the existence of a $s>\bar{t}$ in $A$. Thus $A_k=[0,1]$ for all $k$, so $f^{\pi_n}$ converges uniformly to $f$. In particular $\int f_{-}^{\pi_n} dg \to \int f_{-} dg $ as $n\to \infty$, and since the sequence $(g_k(1))_k$ is converging and thus is bounded, the inequality 
$$ |\sum_{\pi_k}( f^{\pi_n}- f)  \Delta g_k | \leq \sup_s |(f^{\pi_n}- f)(s)|   \sup_k g_k(1)$$
shows that   $\sum_{\pi_k}( f^{\pi_n}- f)  \Delta g_k \to 0$ as $n\to \infty$, uniformly in $k$.
Thus, to conclude the proof it is enough to show \eqref{DiscreteConvCont} when $f$ is replaced by $f^{\pi_n}$.
Since, for $k\geq n$,
$$\int f_{-}^{\pi_n} dg=\sum_{\pi_n} f  \Delta g \quad \text{ and } \quad \sum_{\pi_k} f^{\pi_n} \Delta g_k=\sum_{\pi_n} f^{\pi_n}  \Delta g_k \,=\sum_{\pi_n} f  \Delta g_k \,  , $$
we need to show that $\sum_{\pi_n} f \Delta g_k \to \sum_{\pi_n} f  \Delta g$ as $k\to \infty$, which is trivially true since $g_k(t)\to g(t)$ at every $t\in D\supseteq \pi_n$.
\end{proof}
We will call \emph{optional partition} an increasing  finite or infinite sequence  of stopping times  $\pi$ which is \emph{pointwise finite on compacts}, meaning that  $\pi=(\s_n)_{n\in I}, I\subseteq \N$, $\s_n\leq \s_{n+1}$ for all $n$ and  $\{n: \s_n(\omega)\leq t\}$ is a finite  set for any $\omega$ and $t<\infty$. 
Notice that,  working on the time index $[0,1]$, Theorem \ref{ExPredJumps} will tell us that the jumps of size at least $1/n$ (i.e. the jumps in $(-\infty,-1/n]\cup [1/n,\infty)$) of a c\`adl\`ag process $X$ 
form an optional partition of $[0,1]\cup\{ \infty\}$-valued stopping times. 
 If $\pi=(\s_n)_{n}$,  we will denote write $\pi(\omega)$ for the sequence of reals $(\s_n(\omega))_{n}$. 
Given two finite optional partitions $\pi=(\s_n)_{n=0}^N, \hat{\pi}=(\hat{\s}_j)_{j=0}^J$, a convenient way to construct a finite \emph{increasing} family of stopping times $\pi\cup\hat{\pi}$ which satisfies $(\pi\cup\hat{\pi})(\omega)=\pi(\omega) \cup\hat{\pi} (\omega)$ for all $\omega$ is to 
 define  $\pi\cup\hat{\pi}$ to be the ordered\footnote{One can indeed order this family!} family of stopping times 
$$\s_{1} \wedge \hat{\s}_j 
, \quad  \s_{n}
, \quad  \s_n \vee  (\hat{\s}_j  \wedge \s_{n+1} ) , \quad  \s_N
, \quad \s_N\vee \hat{\s}_j  ,
$$
where $j=0,\ldots, J \, ,\,\,  n=0,\ldots, N-1 $.
If $\pi=(\t_j)_{j=0}^J$ is a finite  partition and $N,B$ are  c\`adl\`ag processes, define
\begin{align*}
 \textstyle 
\sum_{\pi} N_{-} \Delta B:=\sum_{j=1}^J  N_{\t_{j-1}} (B_{\t_j}-B_{\t_{j-1}}),
\end{align*}
which satisfies $(\sum_{\pi} N_{-} \Delta B)(\omega)=\sum_{\pi(\omega)} N_{-}(\omega) \Delta B(\omega)$.
 Given a finite optional partition $\pi=(\s_n)_{n=0}^N$, we will say that  a process $B$   is \emph{$\pi$-predictable} if  $B_{\s_n}$ is $ \F_{\s_{n-1}}$-measurable, $B_{0}$ is $ \F_{0}$-measurable and 
\begin{align}
\label{simple} 
\textstyle
B=\I_{\{0\}} B_0+  \sum_{n= 1}^N \I_{( \s_{n-1} ,\s_n ]} B_{\s_n} \, .
\end{align}
Notice that, if $\alpha, \beta, \gamma$ are stopping times and $A$ is a $\F_{\alpha}$ measurable random variable then
 $C:=A\I_{\{\alpha\leq \gamma\}}$  is  $\F_{ \alpha}\cap  \F_{ \gamma} \subseteq  \F_{\alpha \vee(\beta\wedge \gamma)}$-measurable and
 $$A\I_{(\alpha,\gamma]}=A\I_{(\alpha, \alpha \vee(\beta\wedge \gamma) ]}+C \I_{(\alpha \vee(\beta\wedge \gamma) ,\gamma]} ;$$
so, if $B$ is $\pi$-predictable, it is trivially $(\pi\cup\hat{\pi})$-predictable for any finite optional partition $\hat{\pi}$ (this is why we defined $\pi\cup\hat{\pi}$ as above). 
\begin{theorem}
\label{PredSoNat}
A c\`adl\`ag increasing integrable process $A$ s.t. $A_0=0$ is natural iff it is predictable.
\end{theorem}
\begin{proof}
If $M$ is a  martingale bounded by a constant $C$ and $A=(A_t)_{t\in [0,1]}$ is predictable, let $h,\A^n,N_n$ be as in  Theorem \ref{SumDoob},  let $(\s_i)_i$ be the optional partition of the jumps  of $M$ of size at least $1/n$, and set $\hat{\pi}^n:=(1\wedge \s_i)_i$. 
Let $\pi_k^n=(\t_i)_{i}$ be the finite optional partition \mbox{$\D_{N_n} \cup (1\wedge \s_i)_{i=0,\ldots ,k}$}. Since  $\A^n$ is $\D_{N_n}$-predictable,  it is  $ \pi_k^n$-predictable, and thus $(\A^n_{\t_i})_i$ is a $(\F_{\t_i})_i$-predictable increasing process, and so also a $(\F_{\t_i})_i$-natural process.
Thus 
$\E[M_1 \A^n_1]=\E[\sum_{ \pi_k^n} M_{-}\Delta \A^n]$ holds
 since $(M_{\t_i},\F_{\t_i})_i$ is a bounded martingale.

Now, fix a generic $\omega\in \Omega$ and set  $\pi^n(\omega):= \D_{N_n} \cup \hat{\pi}^n(\omega)$; since $(\s_i)_i$ is pointwise finite on compacts, there exists $k_0=k_0(\omega)$ s.t. $ \pi_{k}^n (\omega)=\pi^n(\omega)$ if $k\geq k_0$. Thus 
 $$\Big(\sum_{ \pi_k^n} M_{-}\Delta \A^n\Big)(\omega) \overset{k\to \infty}{\longrightarrow} \Big(\sum_{ \pi^n} M_{-}\Delta \A^n\Big)(\omega):=\sum_{ \pi^n(\omega)} M_{-}(\omega)\Delta \A^n(\omega) ,$$
the sum on the RHS being well defined, as it is finite for each $\omega$.
Since $(\sum_{ \pi_k^n} M_{-}\Delta \A^n)_k$ is dominated by  $C\A^n_1$, it converges also in  $\Lone$, so
\begin{align}\textstyle
\label{AlmostNat}
\E[M_1 \A^n_1]=\E[\sum_{ \pi^n} M_{-}\Delta \A^n] .
\end{align}
We can now apply Lemma \ref{approxSumInt} and Theorem \ref{SumDoob} 
and obtain that
 $\sum_{\pi^n} M_{-}\Delta \A^n$ (resp. $\A_1^n$) 
 converges $\P$ a.s. to $\int_0^1 M_{s-}dA_s$ (resp. $\A_1$);
 since it is dominated by $Ch$ (resp. $h$), we can pass \eqref{AlmostNat} to the limit and obtain that  $A$ is natural.

Assume now that $A$ is natural, and let $A=M+B$ be its Doob-Meyer decomposition; the c\`adl\`ag increasing integrable process $B$ is predictable, thus natural, and now $A=B$ follows from the uniqueness of the Doob-Meyer decomposition of a submartingale of class D into a martingale plus a  natural process, which is easy to prove  (it follows from \cite[Chapter 1, Theorem 4.10]{KaSh88}).
\end{proof}

\section{Consequences for special semimartingales}
\label{Consequences for special semimartingales}
In this section we show how some well known facts about special semimartingales and predictable processes can be derived as  simple consequences of  Theorem \ref{ExPredJumps} applied to the set $F=(0,\infty)$; in particular, we characterize which special semimartingales $S$ have a continuous compensator.

We will often use without explicit mention the following  trivial consequence of the optional sampling theorem:  if $\t$ is an announceable stopping time and $S=M+A$, where $M$ is a uniformly integrable martingale and $A$ is c\`adl\`ag increasing integrable and s.t. $A_0=0$, then $M_{\t-} \in \Lone$ and $\E[\Delta M_{\t}]=0$, so  $S_{\t-} \in \Lone$ and $\E[\Delta S_{\t}]=\E[\Delta A_{\t}]$. 
Also, we will use without further notice the fact that predictability is preserved by stopping (this follows from $X^\t \I_{\{\t>0\}}=X\I_{(0,\t]}+X_\t \I_{(\t,\infty)}$).

\begin{theorem}
\label{PredMart}
Almost every path of a predictable local martingale $M$ is continuous.
\end{theorem}
\begin{proof}
By localization we can assume that $M$ is a uniformly integrable martingale.
Theorem \ref{ExPredJumps} provides us with a  sequence   $(\s_n)_n$ of predictable stopping times which exactly exhausts the positive jumps of $M$ (i.e. the jumps of $M$ in $(0,\infty)$), and Corollary \ref{PrImplAnn} tells us that  $(\s_n)_n$  are announceable.
It follows that  $\E[M_{\s_n}]=\E[M_{\s_n-}]$ and so, since by definition we have $\Delta M_{\s_n}> 0$ on $\{\s_n<\infty\}$ and $\Delta M_{\s_n}= 0$ on $\{\s_n=\infty\}$,  necessarily each $\{\s_n<\infty\}$ has probability zero.
It follows that $\{ \sup_t \Delta M_t> 0 \}=\cup_n \{ \s_n <\infty\}$ has probability zero, and analogously so does  $\{ \inf_t \Delta M_t<0 \}=\{\sup_t \Delta(- M)_t>0 \}$, so $M$ is a.s. continuous.
\end{proof}

We will say that $M+A$ is a  (canonical) semimartingale decomposition of a process $S$ if   $M$ is a local martingale, $A$ is a c\`adl\`ag adapted (resp. predictable) process of finite variation s.t. $A_0=0$ and $S=M+A$. 
A process $S$ admitting a (canonical) semimartingale decomposition is called a (special) semimartingale. 
Recall that the canonical semimartingale decomposition is unique (for a proof see e.g. \cite[Lemma 22.11]{Kall97}), and that the process $A$  is called the \emph{compensator} of $S$; $M$ is called the (local) martingale part of $S$.

We now need the following theorem, whose elementary proof (which we provide below for convenience of the reader) is essentially standard; the only unconventional choice is to prove it using  Lemma \ref{ApproachTimeLemma}. The advantage of this approach  is that it is much easier to show that $\s$ defined in \eqref{1appr} is a  stopping time than showing that $\s:=\inf\{ t \geq 0:  |X_t|  \geq K \} $ is one (see e.g. \cite[Chapter 2, Lemma 75.1]{RoWi_1:94});
of course it would be even easier to use $\t:=\inf\{ t \geq 0:  |X_t|  > K \} $ instead, but the problem is that it is unclear whether  $\t$ is a \emph{predictable} stopping time when $X$ is predictable.

\begin{theorem}
\label{PredThenLocBdd}
Any c\`adl\`ag predictable process $X$ is locally bounded.
\end{theorem}
\begin{proof}
Given $X$ c\`adl\`ag predictable, let $C_k:=(-\infty,k]\cup [k,\infty)$ and
\begin{align}\label{1apprk}
\s_k:=\inf\{ t \geq 0: X_t  \in C_k \text{ or } X_{t-}  \in C_k \} .
\end{align}
 Lemma \ref{ApproachTimeLemma} and Corollary \ref{PrImplAnn} show  that $\s_k$ is an announceable stopping time.
 Trivially $\s_k\leq \s_{k+1}$; since $X$ is c\`adl\`ag, each of its paths is bounded on compacts, so $\s_k\to \infty$. 
 Let $(\t_k^n)_n$ be a sequence of stopping times announcing $\s_k$, and  $(n_k)_k$ be a subsequence s.t. $\P( \tau_k^{n_k}+1/2^k \leq \s_k\ <\infty)<1/2^k$, so that a.s. $ \tau_k^{n_k}+1/2^k \leq \s_k\ <\infty$ holds for at most finitely many $k$'s, and thus the increasing sequence of stopping times $\rho_i:=\inf_{k\geq i}  \tau_k^{n_k}$ converges to $\lim_k\s_k=\infty$. Since  $|X^{\rho_k}|\I_{\{\rho_k>0\}}\leq k$
holds because $\rho_k\leq \tau_k^{n_k}<\s_k$ on $\{\s_k>0\}$, $X$ is locally bounded. 
\end{proof}

 Here an immediate and useful consequence of Theorem \ref{PredThenLocBdd}.
\begin{corollary}
\label{FinSoLocBdd}
If $A$ is c\`adl\`ag predictable and of finite variation then its variation is  locally bounded.
\end{corollary}
From  Corollary \ref{FinSoLocBdd} it follows that if $S$ is a special semimartingale then one can write $S$ as $M+A$ for a local martingale $M$ and a c\`adl\`ag adapted process $A$ of locally integrable variation (the vice versa is also true, and is given by the Doob-Meyer  decomposition).
Moreover, the optional sampling theorem implies that any local martingale is locally integrable (see \cite[Chapter 3, Theorem 38]{Pr04}), thus
in \emph{any} decomposition of a special semimartingale $S$ as $M+A$, where $A$ is a process of \emph{finite} variation and $M$ is a local martingale, the process $A$ is of  locally integrable variation.
The next important characterization of special semimartingales is also a  consequence of Corollary \ref{FinSoLocBdd}.
For its simple proof we refer to \cite[Chapter 3, Theorem 32]{Pr04}); 
we remark that the proof implicitly makes use of the uniqueness of the canonical decomposition  to obtain the existence of a canonical decomposition of $S$ on $[0,\infty)$ from the ones on $[0,\s_n]$.
\begin{corollary}
\label{ContSpecial}
A semimartingale $S$ is special iff the
process $X_t := \sup_{s\leq t} |\Delta S_s|$ is locally integrable (or equivalently if $S^*_t:=\sup_{s\leq t} |S_s|$ is locally integrable).
\end{corollary}

It follows from Corollary \ref{ContSpecial}  that any continuous semimartingale is special, and then Theorem \ref{PredMart} implies that 
 its  local martingale part and compensator are continuous processes.
More generally, one can characterize which special semimartingales $S$ have a continuous compensator.

\begin{theorem}
\label{Acont}
 If $S=M+A$ is the canonical  decomposition of the special semimartingale $S$, then $A$ is a.s.  continuous iff, for all announceable stopping times $\t$,  $\E[\Delta S^{\s_n}_{\t}]=0$
holds for one (and thus all) sequences of stopping times $\s_n\uparrow \infty$ s.t.
$\I_{\{\s_n>0\}}(\sup_{t\leq \s_n}|M_t| +var(A)_{\s_n}) \in \Lone$.
In particular, if $S=M+A$ is the Doob-Meyer decomposition of a submartingale $S$ of class D, then $A$ is a.s. continuous iff  $\E[\Delta S_{\t}]=0$ for all announceable stopping times $\t$. 
\end{theorem}
\begin{proof} 
One implication is obvious.
For the opposite one, assume by localization that $\sup_t|M_t|$ and $var(A)$ are integrable, and let  $(\t_n)_n$ be a sequence of predictable stopping times which exactly exhausts the positive jumps of $A$.
Since by Corollary \ref{PrImplAnn} predictable stopping times are announceable,  we obtain 
that $0=\E[\Delta S_{\t_n}]=\E[\Delta A_{\t_n}]$. Since $\Delta A_{\t_n}\geq 0$, it follows that  $\Delta A_{\t_n}=0$ a.s. for all $n$, so $\t_n=\infty$ a.s. and $\P(\{\sup_t \Delta A_t >0\})=\sum_n \P(\{\t_n<\infty\})=0$.
Analogously \mbox{$\P(\{\Delta \sup_t(-A)_t >0\})=0$}, so $A$ has a.s. continuous paths.
\end{proof}


\begin{thebibliography}{RW00b}

\bibitem[BSV12]{BeScVe12}
Mathias Beiglboeck, Walter Schachermayer, and Bezirgen Veliyev.
\newblock A short proof of the doob--meyer theorem.
\newblock {\em Stochastic Processes and their Applications}, 122(4):1204--1209,
  2012.

\bibitem[JS03]{JacodShir:02}
Jean Jacod and Albert~N. Shiryaev.
\newblock {\em Limit theorems for stochastic processes}, volume 288 of {\em
  Grundlehren der Mathematischen Wissenschaften [Fundamental Principles of
  Mathematical Sciences]}.
\newblock Springer-Verlag, Berlin, second edition, 2003.

\bibitem[Kal97]{Kall97}
Olav Kallenberg.
\newblock {\em Foundations of modern probability}.
\newblock Probability and its Applications (New York). Springer-Verlag, New
  York, 1997.

\bibitem[KS88]{KaSh88}
I.~Karatzas and S.~Shreve.
\newblock {\em Brownian motion and stochastic calculus}, volume 113 of {\em
  Graduate Texts in Mathematics}.
\newblock Springer-Verlag, New York, 1988.

\bibitem[Low13]{LoBlog}
G.~Lowther.
\newblock {\em
  http://almostsure.wordpress.com/2011/05/26/predictable-stopping-times-2/},
  2013.

\bibitem[L{\v{Z}}07]{LaZi07}
K.~Larsen and G.~{\v{Z}}itkovi{\'c}.
\newblock Stability of utility-maximization in incomplete markets.
\newblock {\em Stochastic Process. Appl.}, 117(11):1642--1662, 2007.

\bibitem[MP80]{MePe80}
Michel M{\'e}tivier and Jean Pellaumail.
\newblock {\em Stochastic integration}, volume 168.
\newblock Academic Press New York, 1980.

\bibitem[Pro04]{Pr04}
P.E. Protter.
\newblock {\em Stochastic integration and differential equations}, volume~21 of
  {\em Applications of Mathematics (New York)}.
\newblock Springer-Verlag, Berlin, second edition, 2004.
\newblock Stochastic Modelling and Applied Probability.

\bibitem[Rao69]{Rao69}
K~Murali Rao.
\newblock On decomposition theorems of meyer.
\newblock {\em Mathematica Scandinavica}, 24:66--78, 1969.

\bibitem[RW00a]{RoWi00}
L.~C.~G. Rogers and D.~Williams.
\newblock {\em Diffusions, Markov processes and martingales: Vol. 2, It{\^o}
  calculus}.
\newblock Cambridge university press, 2000.

\bibitem[RW00b]{RoWi_1:94}
L.~C.~G. Rogers and David Williams.
\newblock {\em Diffusions, {M}arkov processes, and martingales: {V}ol. 1,
  Foundations}.
\newblock Cambridge University Press, 2000.

\end{thebibliography}
\end{document}